\documentclass[12pt]{article}

\usepackage{tikz}
\usepackage{jfe_preambule}
\usepackage{booktabs}
\usepackage{bbm}

\begin{document}

\setlist{noitemsep}  

\title{A few properties of sample variance}

\author{Eric Benhamou \thanks{A.I. SQUARE CONNECT, 35 Boulevard d'Inkermann 92200 Neuilly sur Seine, France}  
\textsuperscript{,} 
\thanks{LAMSADE, Université Paris Dauphine, Place du Maréchal de Lattre de Tassigny,75016 Paris, France} 
\textsuperscript{,} 
\thanks{E-mail: eric.benhamou@aisquareconnect.com, eric.benhamou@dauphine.eu}
}

\date{}              

\singlespacing

\maketitle

\vspace{-.2in}
\begin{abstract}
A basic result is that the sample variance for i.i.d. observations is an unbiased estimator of the variance of the underlying distribution (see for instance \cite{Casella_2002}). But what happens if the observations are neither independent nor identically distributed. What can we say? Can we in particular compute explicitly the first two moments of the sample mean and hence generalize formulae provided in \cite{Tukey_1957a}, \cite{Tukey_1957b}  for the first two moments of the sample variance? We also know that the sample mean and variance are independent if they are computed on an i.i.d. normal distribution. This is one of the underlying assumption to derive the Student distribution \cite{Student_1908}. But does this result hold for any other underlying distribution? Can we still have independent sample mean and variance if the distribution is not normal? This paper precisely answers these questions and extends previous work of \cite{Cho_2004}. We are able to derive a general formula for the first two moments and variance of the sample variance under no specific assumptions. We also provide a faster proof of a seminal result of \cite{Lukacs_1942} by using the log characteristic function of the unbiased sample variance estimator.
\end{abstract}

\medskip



\thispagestyle{empty}

\clearpage

\onehalfspacing
\setcounter{footnote}{0}
\renewcommand{\thefootnote}{\arabic{footnote}}
\setcounter{page}{1}


\section{Introduction}
Let $X_1, \ldots, X_n$ be a random sample and define the sample variance statistic as:
\begin{equation}\label{definition}
\bar{X}_n =\frac{1}{n}\sum_{i=1}^{n}X_i, \quad s_n^2 = \frac{1}{n-1}\sum_{i=1}^{n}(X_i - \bar{X}_n)^2, \quad X_n = (X_1, \ldots, X_n)^T
\end{equation}
where $\bar X_n $ is the empirical mean, $s_n^2$ the empirical Bessel corrected empirical variance also called sample variance, and $X_n$ the vector of the full history of this random sample. 

We are interested in the first two moments of the sample variance as well as its relationship with the sample mean. A basic result is that the sample variance for i.i.d. observations is an unbiased estimator of the variance of the underlying distribution. But what happens if the observations are neither independent nor identically distributed. What can we say?

Can we in particular compute explicitly the first two moments of the sample variance without any particular assumptions on the sample? Can we generalize standard formula for the first two moments of the sample variance as provided in \cite{Tukey_1957a}, \cite{Tukey_1957b}. We also know that the sample mean and variance are independent if they are computed from an i.i.d. normal distribution. But what about any other underlying distribution? Can we still have independent sample mean and variance if the distribution is not normal for an i.i.d. sample? These are the motivations of this paper. This paper extends classical statistical results found in \cite{Cho_2004} but also 
\cite{Tukey_1950}, \cite{Tukey_1956}, \cite{Tukey_1957a}, \cite{Tukey_1957b}. It is organized as follows. First we derive the fist two moments for the sample variance. We then examine the condition for the sample mean and variance to be independent. We show that it is only in the specific case of an underlying normal distribution that they are independent. We conclude on possible extensions.

\section{Moment properties}
\subsection{symmetrical form of the sample variance}
A first property that will be useful in the rest of th paper is the writing of the sample variance as a "U-statistic" (or symmetric) form as given by the following lemma

\begin{lemma}\label{lemma_u_stat}
The sample variance can be defined as the average of the kernel $h(x_1,x_2) = (x_1-x_2)^2/2$ over all $n (n-1)$ pairs of observations $(X_i, X_j)$  for $i\neq j$:
\begin{equation}\label{lemma_eq1}
s_n^2 = \frac{1}{n(n-1)} \sum_{i,j=1}^{n} \frac{(X_i-X_j)^2}{2} = \frac{1}{n(n-1)} \sum_{i \neq j}\frac{(X_i-X_j)^2}{2}
\end{equation}
\end{lemma}

\begin{proof} See proof \ref{proof_u_stat}
\end{proof}

This symmetric form for the sample variance helps us computing the various moments of the sample variance.  Denoting by $\mu_k = \mathbb{E}\left[ X^k \right]$ the various moment of the variable $X$ and assuming that $(X_i)_{i=1,\ldots, n}$ are n observations of the variable $X$ (not necessarily independent), we can start computing the sample variance moments.

\subsection{First moment of sample variance}
\begin{lemma}\label{moments_1}
The expectation of the sample variance is given by:
\begin{equation}\label{lemma_mom_eq1}
\mathbb{E}\left[ s_n^2 \right]  =  \frac{\sum_{i=1}^n \mathbb{E} \left[ X_i^2 \right] }{n}  - \frac{  \sum_{i \neq j} \mathbb{E}\left[X_i X_j \right] }{n(n-1)} 
\end{equation}
Hence if $(X_i)_{i=1,\ldots,n}$ is independent and identically distributed, we get that $s_n^2$ is an unbiased estimator of the variance :
\begin{equation}\label{lemma_mom_eq2}
\mathbb{E}\left[ s_n^2 \right]= \mu_2 -  \mu_1^2 = \text{Var}[X]
\end{equation}
where $\mu_2 = \mathbb{E}\left[ X^2 \right]$ and $\mu_1 = \mathbb{E}\left[ X \right]$.
\end{lemma}

\begin{proof} See proof \ref{proof_mom_1}
\end{proof}

This lemma calls various remarks. First of all, the fact that for iid sample, the sample variance is unbiased is very well know (see for instance \cite{Casella_2002}). Secondly, the cross term $\mathbb{E}\left[X_i X_j \right]$ implies that this estimator will not be unbiased for correlated sample as the expectation can rewite as
\begin{equation}
\mathbb{E}\left[ s_n^2 \right]  =  \frac{\sum_{i=1}^n \mathbb{E} \left[ X_i^2 \right] }{n}  - \frac{  \sum_{i \neq j} \mathbb{E}\left[X_i  \right] \mathbb{E}\left[ X_j \right] }{n(n-1)} -  \frac{  \sum_{i \neq j} \left( \mathbb{E}\left[X_i X_j \right] -\mathbb{E}\left[X_i  \right] \mathbb{E}\left[ X_j \right] \right)  }{n(n-1)} 
\end{equation}
Hence for a non independent sample the term $\mathbb{E}\left[X_i X_j \right] -\mathbb{E}\left[X_i  \right] \mathbb{E}\left[ X_j \right] $ does not cancel, while the first and second terms can be interpreted as estimator of the second and first moment of the sample. More generally, the interest of this general lemma is its application to non independent and non identically distributed samples. 

\subsection{application to AR(1)}
Let us apply our result to a non independent sample. For instance, assume that the sample $(X_i)_{i=1,\ldots,n}$ is generated by an auto regressive process of order 1 (AR(1)). We impose that the process is stationary with mean 0, variance $\frac{\sigma^2 }{1-\rho^2}$ where $\sigma$ is the variance of the underlying noise and $\rho$ is the first order correlation. In this specific case, our general formula provides the expectation of the sample variance. We find that the sample variance is biased and given by
\begin{equation}
\mathbb{E}\left[ s_n^2 \right] = \frac{\sigma^2 }{1-\rho^2} \left( 1 - \frac{ 2 \rho}{(1-\rho)(n-1)} + \frac{2 \rho (1-\rho^n) }{n (n-1)(1-\rho^2)}\right)
\end{equation}

\begin{proof} See proof \ref{proof_mom_1_app}
\end{proof}

\subsection{Second moment of sample variance}
\begin{lemma}\label{moments_2}
The second moment of the sample variance is given by:
\begin{equation}\label{lemma_mom_mom2_1}
\mathbb{E}\left[ s_n^4 \right]=\frac{\mathbb{E}[\hat\mu_4]}{n} - \frac{4 \mathbb{E}[\hat \mu_1 \hat \mu_3]}{n} + \frac{ (n^2-2 n+3)\mathbb{E}[\hat \mu_2^2]}{n(n-1)} - \frac{2(n-2)(n-3)\mathbb{E}[\hat\mu_1^2 \hat\mu_2]}{n(n-1)} + \frac{(n-2)(n-3) \mathbb{E}[\hat\mu_1^4]}{n(n-1)}
\end{equation}

where we have adopted the following notation:
\begin{align} \label{notations}
\hat{\mu}_4 &= \frac{\sum_{i=1}^n X_i^4 }{n} & \hat{\mu}_3 \hat{\mu}_1 &= \frac{ \sum_{i \neq j } X_i^3 X_j }{n (n-1)}  & \hat{\mu}_2 ^2 & = \frac{ \sum_{i \neq j} X_i^2 X_j^2 }{n (n-1)} \nonumber \\
\hat{\mu}_1^2 \hat{\mu}_2 & = \frac{\sum_{i \neq j \neq k} X_i^2 X_j X_k }{n (n-1)(n-2) } & \hat{\mu}_1^4 & =  \frac{\sum_{i \neq j \neq k \neq l} X_i X_j X_k  X_l }{n (n-1)(n-2) (n-3)} &  
\end{align}
\end{lemma}

\begin{proof} See proof \ref{proof_mom_2}
\end{proof}

Like previously, the expression for the second moment of the sample mean is very general and an extension of previous results. Its interest is precisely to apply without any restriction on the underlying observation. This generalizes in particular \cite{Cho_2004}, but also \cite{Tukey_1950}, \cite{Tukey_1956}, \cite{Tukey_1957a}, and \cite{Tukey_1957b}.

\subsection{Variance of sample variance}
\begin{lemma}\label{lemma_mom_mom2_22}
The variance of the sample variance is given by:
\begin{eqnarray}
\text{Var}\left[ s_n^2 \right]& =& \frac{\mathbb{E}[\hat\mu_4]}{n} - \frac{4 \mathbb{E}[\hat \mu_1 \hat \mu_3]}{n} + \frac{ (n^2-2 n+3)\mathbb{E}[\hat \mu_2^2]}{n(n-1)} - ( \mathbb{E}[\tilde{\mu}_2])^2   \nonumber \\
& & \quad \quad - \frac{2(n-2)(n-3)\mathbb{E}[\hat\mu_1^2 \hat\mu_2]}{n(n-1)}  +2   \mathbb{E}[\tilde{\mu}_1^2]  \mathbb{E}[\tilde{\mu}_2] \nonumber  \\
& & \quad \quad + \frac{(n-2)(n-3) \mathbb{E}[\hat\mu_1^4]}{n(n-1)} -  (\mathbb{E}[\tilde{\mu}_1^2])^2 \label{lemma_mom_eq22}
\end{eqnarray}
where $\tilde{\mu}_2  = \cfrac{ \sum_{i} X_i^2 }{n}  $ and $ \tilde{\mu}_1^2  =  \cfrac{\sum_{i \neq j } X_i X_j }{n (n-1)}$.  If the observations $(X_i)_{i=1,\ldots,n}$ are independent and identically distributed and if we denote by $\mu_i^c$ the central moments of the population, this simplifies into:
\begin{equation}
\text{Var}\left[ s_n^2 \right]= \frac{ \mu_4^c}{n} -\frac{ (n-3) (\mu_2^c)^2 }{n(n-1)} 
\end{equation}
If the observations are from an i.i.d. normal distribution, this results in the traditional result
\begin{equation}
\text{Var}\left[ s_n^2 \right]= \frac{ 2 \sigma^4 }{n-1} 
\end{equation}

\end{lemma}

\begin{proof} See proof \ref{proof_mom_22}
\end{proof}

Like previous results, equation (\ref{lemma_mom_eq22}) is the most general one and encompasses cases where observations are not necessarily independent nor identically distributed. To our knowledge, these results are new and give as a byproduct all standard results about the first, second and variance of the sample mean that can be found in textbook like \cite{Casella_2002}.

\section{Relationship between sample mean and variance}
We finally tackle the question of the condition for the sample mean and variance to be independent. This is a strong result that for instance enables us to derive the Student distribution as in the normal case of iid variables, the sample mean and variance are clearly independent. We are interested in the opposite. What is the condition to impose on our distribution for iid variable to make our sample mean and variance independent? We shall prove that it is only in the case of normal distribution that these two estimators are independent as the following proposal states

\begin{proposition}
The sample mean and variance are independent if and only if the underlying (parent) distribution is normal.
\end{proposition}

\begin{proof} This result was first proved by \cite{Geary_1936} and later by \cite{Lukacs_1942}. We provide in \ref{normal} a proof that uses modern notations. It is an adaptation of the  proof  in \cite{Lukacs_1942} but with a simpler approach as we work with the log characteristic function and the unbiased sample variance. This makes the resulting differential equation trivial to solve as this is just a constant second order derivative constraint.
\end{proof}

This result implies consequently that it will not easy to derive the underlying distribution of the t-statistic for a non normal distribution. Indeed the t-statistic is defined as the ratio of the sample mean over the sample variance. If the sample mean and sample variance are not independent, the computation of the underlying distribution does not decouple. This makes the problem of the computation of the underlying distribution an integration problem that has no closed form. This kills in particular our hope to derive other distribution that generalizes the case of the Student distribution to non normal underlying assumptions.


\section{Conclusion}
In this paper, we have derived the most general formula for the fist, second moment and variance of the sample variance. Our formula does not assume that the underlying sample is independent neither identically distributed. We also show that for an i.i.d. sample, the independence between the sample mean and variance is characteristic of the normal distribution. Possible extensions are to computer higher moments for the sample variance.

\clearpage

\appendix
\section{Symmetry for the sample variance}\label{proof_u_stat}
Let us first prove equation
\begin{equation}\label{lemma_eq2}
 \frac{1}{n-1}\sum_{i=1}^{n}(X_i - \bar{X}_n)^2= \frac{1}{n-1}\left( \sum_{i=1}^{n}X_i^2 - n  \bar{X}_n^2 \right)
\end{equation}

with $\bar{X}_n$  defined by equation (\ref{definition}). Expanding the left hand side (LHS) leads to
\begin{eqnarray}
 LHS & =&  \frac{1}{n-1}\sum_{i=1}^{n} ( X_i^2 + \bar{X}_n^2 - 2 X_i \bar{X}_n ) \\
& = & \frac{1}{n-1} \left( \sum_{i=1}^{n} X_i^2 + n \bar{X}_n^2 - 2 \bar{X}_n \sum_{i=1}^{n} X_i \right) \\
& = & \frac{1}{n-1} \left( \sum_{i=1}^{n} X_i^2 - n \bar{X}_n^2 \right)
\end{eqnarray}

We want to prove equation (\ref{lemma_eq1}):
\begin{equation}
 \frac{1}{n-1}\sum_{i=1}^{n}(X_i - \bar{X}_n)^2= \frac{1}{n(n-1)}\sum_{i,j=1}^{n}\frac{(X_i-X_j)^2}{2} = \frac{1}{n(n-1)} \sum_{i \neq j}\frac{(X_i-X_j)^2}{2},
\end{equation}

Note that the forms where we sum over all pairs of $i,j$ and where we sum over all pairs that are different $i \neq j$ are equal (middle and right hand side) as the missing terms between the two sides are equal to zero. 
Some routine algebraic reduction on the middle hand side (MHS) gives:
\begin{eqnarray}
MHS & =& \frac{1}{n(n-1)}\sum_{i,j=1}^{n}\frac{X_i^2+X_j^2 - 2X_i X_j}{2} \\
& =&  \frac{1}{n(n-1)} \frac{2 n \sum_{i=1}^{n}X_i^2} {2} - \frac{1}{n(n-1)} \sum_{i,j=1}^{n} X_i X_j \\
& =&  \frac{1}{n-1}\sum_{i=1}^{n}X_i^2 - \frac{1}{n(n-1)} (n \bar{X}_n)^2 \\
& =&  \frac{1}{n-1} \left( \sum_{i=1}^{n}X_i^2 - n \bar{X}_n^2  \right) 
\end{eqnarray}

We can easily conclude using equation  (\ref{lemma_eq2})
\qed

\section{Moment of sample variance}
\subsection{First moment of sample variance}\label{proof_mom_1}
The result of lemma \ref{moments_1} is immediate expanding lemma \ref{lemma_u_stat} equality:
\begin{eqnarray}
s_n^2 &= & \frac{1}{n(n-1)} \sum_{i \neq j}\frac{(X_i-X_j)^2}{2} \\
&= & \frac{1}{n(n-1)} \left( (n-1) \sum_{i=1}^n X_i^2 - \sum_{i \neq j} X_i X_j  \right) ,
\end{eqnarray}

and taking the expectation. The case of i.i.d. variables is also trivial as independence implies
$$
\mathbb{E}\left[X_i X_j \right] = \mathbb{E}\left[X_i \right] \mathbb{E}\left[ X_j \right]  
$$

The identically distributed assumption implies $\mathbb{E} \left[ X_i^2 \right] = \mu_2$ and  $\mathbb{E}\left[ X_i \right] = \mu_1$. Finally, we have 
$$
\sum_{i=1}^n \mathbb{E} \left[ X_i^2 \right] = n \mu_2, \quad \text{and} \quad  \sum_{i \neq j}  \mathbb{E}\left[ X_i \right]   \mathbb{E}\left[ X_j \right] = n(n-1) \mu_1^2
$$
\qed.

\subsection{Application to AR(1)}\label{proof_mom_1_app}
Lemma \ref{moments_1} can be rewritten as 
\begin{equation}
\mathbb{E}\left[ s_n^2 \right] = \frac{\sum_{i=1}^n \mathbb{E} \left[ X_i^2 \right] }{n} -  \frac{  \sum_{i \neq j}  \mathbb{E}\left[ X_i \right]   \mathbb{E}\left[ X_j \right]  }{n(n-1)} - \frac{1}{n(n-1)} \sum_{i \neq j} \left(\mathbb{E}\left[X_i X_j \right] - \mathbb{E}\left[X_i \right] \mathbb{E}\left[ X_j \right] \right)
\end{equation}

In the case of an AR(1) process, we have
\begin{equation}
\mathbb{E}\left[X_i X_j \right] - \mathbb{E}\left[X_i \right] \mathbb{E}\left[ X_j \right]  =\frac{\sigma^2 }{1-\rho^2} \rho^{ \vert i-j  \vert}
\end{equation}

Hence, the term due to non independent is computed as follows:
\begin{eqnarray}
 \frac{1}{n(n-1)} \sum_{i \neq j} \left(\mathbb{E}\left[X_i X_j \right] - \mathbb{E}\left[X_i \right] \mathbb{E}\left[ X_j \right] \right) &= & \frac{\sigma^2 }{1-\rho^2}  \frac{ \sum_{i \neq j} \rho^{ \vert i-j \vert} }{n(n-1)}
\end{eqnarray}

To conclude, one can use that 
\begin{align}
\sum_{i=1}^{n} (n-i ) \rho^{i} = \rho \frac{ n (1-\rho) - (1-\rho^{n}) }{(1-\rho)^2}
\end{align}

\qed.

\subsection{Second moment of sample variance}\label{proof_mom_2}
Let us do some routine algebraic computation. We have
\begin{eqnarray}
s_n^4 &= & \frac{1}{n^2(n-1)^2} \left( (n-1) \sum_{i=1}^n X_i^2 - \sum_{i \neq j} X_i X_j  \right)^2 \\
& =& \frac{1}{n^2(n-1)^2} \left( (n-1)^2 (\sum_{i=1}^n X_i^2)^2 + (\sum_{i \neq j} X_i X_j)^2 -2 (n-1) (\sum_{i=1}^n X_i^2 )( \sum_{k \neq l} X_k X_l ) \right) \nonumber
\end{eqnarray}

Let us expand. The first expansion $(\sum_{i=1}^n X_i^2)^2 $ is easy and immediate:
\begin{eqnarray}
(\sum_{i=1}^n X_i^2)^2 & = &\sum_{i=1}^n X_i^4 + \sum_{i \neq j} X_i^2 X_j^2 
\end{eqnarray}

In the expansion of $(\sum_{i \neq j} X_i X_j)^2$ we have that the squared terms are with same indexes ($(i \neq j) = (k \neq l)$). The cross terms are $(i \neq j), (k \neq l)$ with the constraint that they are different ($(i,j) \neq (k,l)$). There are three possibilities for these cross terms. These cross terms can either be only two real indexes $(i \neq j)$ and $(j \neq i)$ or vice versa leading to two times the squared terms, or we have that $(i,j,k,l)$ are in fact only three numbers and this can happen 4 times as it is either $j,j,k$ or $l$ that coincides with the other indexes or there all different, and this can happen only once. Hence, we have:
\begin{eqnarray}
(\sum_{i \neq j} X_i X_j)^2& = & 3 \sum_{i \neq j} X_i^2 X_j^2 + 4 \sum_{i \neq j \neq k } X_i^2 X_j X_k +  \sum_{i \neq j \neq k \neq l}  X_i X_j X_k X_l  
\end{eqnarray}

To expand $(\sum_{i=1}^n X_i^2 )( \sum_{k \neq l} X_k X_l ) $, we can notice that either there is no intersection of indexes between $i,j$ and $k$, or $i$ coincides with either $k$ or $l$. And this can happen 4 times. So the expansion is given by
\begin{eqnarray}
(\sum_{i=1}^n X_i^2 )( \sum_{k \neq l} X_k X_l )  & = & \sum_{i \neq j \neq k } X_i^2 X_j X_k  + 4 \sum_{i \neq j } X_i^3 X_j
\end{eqnarray}

Regrouping terms leads to
\begin{eqnarray}\label{int_result}
s_n^4 &= & \frac{\frac{\sum_{i=1}^n X_i^4 }{n}  }{n}  -  \frac{  \frac{ 4 \sum_{i \neq j } X_i^3 X_j }{n (n-1)}}{n} + \frac{\frac{(3 + (n-1)^2) \sum_{i \neq j} X_i^2 X_j^2 }{n (n-1)}}{n(n-1)} \nonumber \\
& & \quad - \frac{\frac{(2 (n-1)(n-2) - 4 (n-2)) \sum_{i \neq j \neq k} X_i^2 X_j X_k }{n (n-1)(n-2) }}{n(n-1)}  
+ \frac{\frac{(n-2)(n-3) \sum_{i \neq j \neq k \neq l} X_i X_j X_k  X_l }{n (n-1)(n-2) (n-3)}}{n(n-1)}
\end{eqnarray}

We can conclude by using the notation given in \ref{notations} and taking the expectation \qed

\subsection{Variance of sample variance}\label{proof_mom_22}
The variance of the sample variance writes 
\begin{eqnarray}
\text{Var}[s_n^2] = \mathbb{E}[s_n^4] - (\mathbb{E}[s_n^2])^2.
\end{eqnarray}

Using lemma \ref{moments_1}, the square of the expectation of $s_n^2$ writes:

\begin{eqnarray}
(\mathbb{E}\left[ s_n^2 \right] )^2  & = & \left(  \mathbb{E} \left[ \frac{\sum_{i=1}^n X_i^2 }{n} \right] - \mathbb{E}\left[ \frac{  \sum_{i \neq j}X_i X_j  }{n(n-1)} \right] \right)^2 \\
& = & \left(\mathbb{E} \left[ \frac{\sum_{i=1}^n X_i^2 }{n} \right] \right)^2 -2 \mathbb{E}\left[ \frac{  \sum_{i \neq j}X_i X_j  }{n(n-1)} \right] \mathbb{E} \left[ \frac{\sum_{i=1}^n X_i^2 }{n} \right]  \nonumber \\
&  &  \quad\quad\quad\quad\quad+  \left(\mathbb{E}\left[ \frac{  \sum_{i \neq j}X_i X_j  }{n(n-1)} \right]\right)^2
\end{eqnarray}

Combining this result with \ref{moments_2}, leads to equation (\ref{lemma_mom_eq22}). To use central moment, it is equivalent to impose that observations have null odd moments. Adding the condition of i.i.d observations implies that the following equalities:

\begin{align} 
\hat{\mu}_4 &= \mu_4^c & \hat{\mu}_3 \hat{\mu}_1 &= 0 & \hat{\mu}_2 ^2 & = (\mu_2^c)^2 \\
\tilde{\mu}_2 & = \mu_2^c &  \tilde{\mu}_1^2 & =  0 & \hat{\mu}_1^2 \hat{\mu}_2 & = 0 & \hat{\mu}_1^4 & =  0
\end{align}

Hence, we get 
\begin{equation}
\text{Var}\left[ s_n^2 \right]= \frac{ \mu_4^c}{n} -\frac{ (n-3) (\mu_2^c)^2 }{n(n-1)} 
\end{equation}

If the observations are from an i.i.d. normal distribution with zero mean and a variance $\sigma^2$, we have 
$\mu_4^c  =3 \sigma^4 $ and $ \mu_2^c  = \sigma^2$ which leads to the result \qed.

\clearpage

\section{Proof of the condition for sample mean and variance to be independent}\label{normal} 
The assumption of i.i.d. sample for  $(x_1,\ldots, c_n)$ implies that the joint distribution of $(x_1,\ldots, x_n)$ denoted by
$f_{X_1,\ldots,X_n}(x_1,\ldots, x_n)$ is equal to $\prod_{i=1}^n f_X(x_i)$, which we will write $\prod_{i=1}^n f(x_i)$ dropping the $._{X}$ to make notation lighter.

The log of the characteristic function of the joint variable $(\bar X_n,s_n^2)$ writes
\begin{equation}
\ln( \phi_{(\bar X_n,s_n^2)}(t_1, t_2) ) = \ln \left( \iiint e^{i t_1 \bar x_n + i t_2 s_n^2} \prod_{i=1}^n f(x_i) dx_i \right).
\end{equation}

Similarly, the log of the characteristic function for the sample mean $\bar X_n$  writes 
\begin{equation}
\ln( \phi_{\bar X_n} (t_1) ) =\ln \left(  \iiint e^{i t_1 \bar x_n } \prod_{i=1}^n f(x_i) dx_i \right),
\end{equation}
and similarly for the sample variance
\begin{equation}
\ln( \phi_{s_n^2}(t_2) )= \ln \left(  \iiint e^{ i t_2 s_n^2} \prod_{i=1}^n f(x_i) dx_i \right).
\end{equation}

The assumption of independence between  sample mean $\bar X_n$ and variance $s_n^2$ is equivalent to the fact that the characteristic function of the couple decouples, or that the log characteristic functions sum up.

\begin{equation} \label{condition_1}
\ln( \phi_{(\bar X_n,s_n^2)}(t_1, t_2) ) = \ln( \phi_{\bar X_n} (t_1) ) + \ln( \phi_{s_n^2}(t_2) ).
\end{equation}

Differentiating condition \ref{condition_1} with respect to $t_2$ in $t_2=0$ leads to
\begin{equation} \label{condition_2}
\frac{1}{\phi_{(\bar X_n,s_n^2)}(t_1, t2)} \left. \frac{ \partial \phi_{(\bar X_n,s_n^2)}(t_1, t_2) }{\partial t_2} \right|_{t_2=0} =\frac{1}{\phi_{s_n^2}(t2)}  \left. \frac{ \partial   \phi_{s_n^2}(t_2) }{\partial t_2}  \right|_{t_2=0}.
\end{equation}

Noticing that $\phi_{s_n^2}(0) = 1$ and $\phi_{(\bar X_n,s_n^2)}(t_1, 0) = \phi_{\bar X_n} (t_1)$, the condition \ref{condition_1} writes
\begin{equation} \label{condition_2}
\frac{1}{\phi_{\bar X_n}(t_1)} \left. \frac{ \phi_{(\bar X_n,s_n^2)}(t_1, t_2) }{\partial t_2} \right|_{t_2=0} =\left. \frac{ \partial   \phi_{s_n^2}(t_2) }{\partial t_2}  \right|_{t_2=0}.
\end{equation}

Using the fact that $\bar X_n= \frac{ \sum_{i=1}^n X_i}{n}$, it is easy to see that 

\begin{equation}
 \phi_{\bar X_n} (t_1) = \prod _{i=1}^n \int e^{i t_1 x_i / n} f(x_i) dx_i = [\phi_{X} (t_1 / n)] ^n
\end{equation}

For the sample variance, we can use the "U-statistic" (or symmetric) form as shown in lemma \ref{lemma_u_stat}, to see that 
\begin{equation}
s_n^2 = \frac{\sum_{i =1}^n X_i^2} {n}  - \frac{\sum_{i \neq j} X_i X_j}{n (n-1)}
\end{equation}

Hence, the derivative of the characteristic function of the couple $(\bar X_n, s_n^2)$ writes 
\begin{align}
\left. \frac{ \partial \phi_{(\bar X_n,s_n^2)}(t_1, t_2) }{\partial t_2} \right|_{t_2=0}  & = \iiint i s_n^2  \prod_{i=1}^n e ^{ i t_1 x_i / n } f(x_i) dx_i   \\
& = i \iiint \left( \frac{\sum_{i =1}^n x_i^2} {n}-  \frac{\sum_{i \neq j} x_i x_j}{n (n-1)}  \right)  \prod_{i=1}^n e ^{ i t_1 x_i / n } f(x_i) dx_i   \\
& = i  [\phi_{X}(\frac{t_1}{n})] ^{n-2} \left( \phi_{X} (\frac{t_1}{n}) \int x^2 e ^{\frac{i  t_1 x}{n}} f(x) dx - (\int x^2 e ^{\frac{i t_1 x}{n}} f(x) dx )^2  \right) 
\end{align}

In the latter equation, if we set $t_1=0$, we get in particular that 
\begin{align}
\left. \frac{ \partial \phi_{s_n^2}(t_1, t_2) }{\partial t_2} \right|_{t_2=0}  & =\left. \frac{ \partial \phi_{(\bar X_n,s_n^2)}(0, t_2) }{\partial t_2} \right|_{t_2=0}  =i \sigma^2
\end{align}

Hence, condition (\ref{condition_2}) writes
\begin{equation} \label{condition_3}
\frac{ \phi_{X} (\frac{t_1}{n}) \int x^2 e ^{\frac{i t_1 x}{n}} f(x) dx - (\int x^2 e ^{\frac{i t_1 x}{n}} f(x) dx )^2 }{[\phi_{X}(\frac{t_1}{n})]^2} = \sigma^2
\end{equation}

We also have that the derivative of the characteristic function $\phi_{X} (\frac{t_1}{n})$ with respect to $u = t_1 / n$ gives
\begin{align}\label{derivative_2}
\frac{ \partial \phi_{X} (u)}{\partial u} & =  \int  i x e ^{i x u} f(x) dx
\end{align}

To simplify notation, we drop the index in $\phi_X$ and writes this function $\phi$. Using equation  (\ref{derivative_2}), condition (\ref{condition_3}) writes 
\begin{align}\label{condition_4}
\frac{ - \phi (u) \frac{ \partial^2 \phi (u)}{\partial u^2} + \left( \frac{ \partial \phi (u)}{\partial u} \right)^2}{\phi (u)^2}  = \sigma^2
\end{align}

The log of the characteristic function of $\phi(u) = \mathbb{E}[ e^{i u X} ]$, denoted by $\Psi(u) =\ln \phi (u)$, first and second derivatives with respect to $u$ are given by:
\begin{align}
\frac{ \partial \Psi(u)}{\partial u} & = \frac{ \partial \ln \phi (u)}{\partial u}  = \frac{1}{ \phi (u)} \frac{ \partial  \phi (u)}{\partial u} &  \quad\\
\frac{ \partial^2 \Psi(u)}{\partial^2 u} & =   \frac{ \partial }{\partial u}  \frac{ \partial \Psi(u)}{\partial u} =  \frac{1}{ \phi (u)} \frac{ \partial^2  \phi (u)}{\partial u^2}  - \frac{1}{ \phi (u)^2 }  \left( \frac{ \partial  \phi (u)}{\partial u} \right)^2 &  \quad
\end{align}

Hence,  condition (\ref{condition_4}) writes 
\begin{align}\label{condition_5}
\frac{ \partial^2 \Psi(u)}{\partial^2 u}   = -\sigma^2
\end{align}

Using the boundary conditions $ \Psi(0)= 0$ and $\Psi^{'}(0)= i  \mathbb{E}[X] = i \mu$, it is easy to integrate condition \ref{condition_5} which is a constant second order derivative to get
\begin{align}\label{condition_6}
\Psi(u)  =i  \mu u  - \frac{ \sigma^2 u^2}{2}
\end{align}

Condition \ref{condition_6} states that a necessary and sufficient condition for the sample mean and variance to be independent is that the log characteristic function of $X$ is a quadratic form. But a quadratic form for the log characteristic function of $X$  is a characterization of a normal distribution, which concludes the proof. \qed

\clearpage

\bibliographystyle{jfe}
\bibliography{mybib}

\end{document}